\documentclass[12pt,oneside,reqno]{amsart}
\allowdisplaybreaks

\usepackage[top=3cm, bottom=3cm, left=3cm, right=3cm]{geometry}
\usepackage{multirow,tabularx}
\newcolumntype{Y}{>{\centering\arraybackslash}X}

\usepackage{pbox}
\usepackage{parskip}
\usepackage{amssymb}
\usepackage{hyperref}
\usepackage{mathtools}

\usepackage{enumerate}
\usepackage{tikz-cd}
\usepackage[cmtip,all]{xy}
\newcommand{\longsquiggly}{\xymatrix{{}\ar@{~>}[r]&{}}}

\newtheorem{thm}{Theorem}[section]

\newtheorem{lem}[thm]{Lemma}
\newtheorem{prop}[thm]{Proposition}
\theoremstyle{definition}
\newtheorem{defn}[thm]{Definition}
\theoremstyle{remark}

\newtheorem{ex}{Example}
\numberwithin{equation}{section}

\makeatletter
\newcommand{\BIG}{\bBigg@{2}}
\newcommand{\vast}{\bBigg@{3}}
\newcommand{\Vast}{\bBigg@{5}}
\makeatother

\numberwithin{equation}{section}

\begin{document}


\baselineskip=17pt



\title[Skew-symmetric elements of $FG$]
{Skew-Symmetric Elements of Rational Group Algebras}

\author[D. Chaudhuri]{Dishari Chaudhuri}

\address{
	Department of Mathematical Sciences\\ Indian Institute of Science Education and Research Mohali\\
	Sector-81, Knowledge City, S.A.S. Nagar, Mohali-140306\\ Punjab, India}

\email{dishari@iitg.ac.in, dishari.chaudhuri@gmail.com}

\thanks{The author was supported by DAE (Government of India) and National Board for Higher Mathematics with reference number $2/40(16)/2016/$ R\&D-II/$5766$ during this project. The author would like to thank IISER Mohali for providing good research facilities when this project was carried out. The author is very grateful to Abhay Soman for many insightful discussions. Finally the author would like to thank the unknown referee whose valuable comments and suggestions helped a great deal in improving the exposition of this article.}




\begin{abstract}
	Let $RG$ be the group ring of a finite group $G$ over a commutative ring $R$ with $1$. An element $x$ in $RG$ is said to be skew-symmetric with respect to an involution $\sigma$ of $RG$ if $\sigma(x)=-x.$ A structure theorem for the skew-symmetric elements of $FG$ is given where $F$ is an algebraic extension of $\mathbb{Q}$ which generalizes some previously known results in this direction.
\end{abstract}

\maketitle
\section{Introduction}
Let $RG$ be the group ring of a finite group $G$ over a commutative ring $R$ with $1$ and let $\sigma$ be an involution of $RG$. That is, $\sigma: RG\rightarrow RG$ is such that for $x,y\in RG$, $\sigma$ satisfies the following conditions: $(i)$ $\sigma(x+y)=\sigma(x)+\sigma(y),$ $(ii)$ $\sigma(xy)=\sigma(y)\sigma(x)$ and $(iii)$ $\sigma^2(x)=x.$ Let $RG^+_\sigma=\{\gamma\in RG\;|\;\sigma(\gamma)=\gamma\}$ and $RG^-_\sigma=\{\gamma\in RG\;|\;\sigma(\gamma)=-\gamma\}$ be the set of symmetric and skew-symmetric elements of $RG$ respectively with respect to the involution $\sigma$. If $\sigma$ is an $R$-linear extension of an involution of $G$, then $RG^-_\sigma$ is generated by $\{g-\sigma(g)\;|\;g\in G\backslash G_\sigma\}\cup\{rg\;|\;r\in R_2,g\in G_\sigma\}$ as an $R$-module, where $G_\sigma=\{g\in G\;|\;\sigma(g)=g\}$ and $R_2=\{r\in R\;|\;2r=0\}$. Note that $RG^-_\sigma$ may not be a subring of $RG$ in general. Now, $RG$ may be viewed as a Lie ring with the help of the bracket operation $[x,y]=xy-yx$ for $x,y\in RG$. Then, $RG^-_\sigma$ becomes a Lie subring of $RG$. There are some strong relations between the structure of $RG^-_\sigma$ and $RG$ as Lie rings which have been studied by a few authors (for example,\cite{Amitsur-69}, \cite{Zaleskii-Smirnov-81},\cite{Giambruno-Sehgal-93}). Also there are some close relations between polynomial identities satisfied by the unitary units of $RG$ and the Lie algebra $RG^-_\sigma$ (\cite{Giambruno-Milies-03}). The question of when $RG^-_\sigma$ is commutative has been completely answered in terms of the group elements of $G$ (\cite{Jespers-Ruiz-05}, \cite{Jespers-Ruiz-06}, \cite{Cristo-Milies-07}, \cite{CJMM-09}). Lie properties of $RG^-_\sigma$ has been studied as well (\cite{Lee-Sehgal-Spinelli-09}). So we can see that the Lie algebra $RG^-_\sigma$ has been a useful object of study in quite a few cases. Hence it is interesting to investigate the structure of $RG^-_\sigma$ itself.\\

Given a semisimple group algebra $RG$, the Lie algebra $RG_\sigma^-$ will be semisimple and hence it is natural to seek a concrete description of the simple Lie algebras into which it decomposes as a direct sum. When $R=\mathbb{C}$ and $\sigma$ is the canonical involution of $\mathbb{C}G$, that is, the $\mathbb{C}$-linear extension of $g\mapsto g^{-1}$, the structure of $\mathbb{C}G^-_\sigma$ is given as a direct sum of classical simple Lie algebras by Cohen and Taylor \cite{CT-07}. The Lie algebra $\mathfrak{gl}(n)$ is the space of all linear transformations of $\mathbb{C}^n$, where the Lie product is the usual bracket operation. If $n\geq1$ and if $b$ is a non-degenerate alternating or symmetric bilinear form on $\mathbb{C}^n$, then the subspace of $\mathfrak{gl}(n)$ consisting of all $x$ such that $b(xu,v)+b(u,xv)=0$ for all $u,v\in\mathbb{C}^n$ is a Lie algebra. When $b$ is skew-symmetric, $n$ is necessarily even, and we have the \emph{symplectic} Lie algebra $\mathfrak{sp}(n)$. When $b$ is symmetric, we have the \emph{orthogonal} Lie algebra $\mathfrak{o}(n)$. Let $\chi$ be a character of $G$. It was shown in \cite{CT-07} that $\mathbb{C}G^-_\sigma$ admits the Lie algebra decomposition:

\begin{equation}
\mathbb{C}G^-_\sigma\simeq\bigoplus_{\chi\in\mathfrak{R}}\mathfrak{o}(\chi(1))\oplus\bigoplus_{\chi\in\mathfrak{Sp}}\mathfrak{sp}(\chi(1))\oplus\bigoplus_{\chi\in\mathfrak{C}}{'}\mathfrak{gl}(\chi(1))
\end{equation}

\noindent where $\mathfrak{R}, \mathfrak{Sp}$ and $\mathfrak{C}$ are the sets of irreducible characters of $G$ of real, symplectic and complex types, respectively. The prime signifies that there is just one summand $\mathfrak{gl}(\chi(1))$ for each pair $\{\chi,\overline{\chi}\}$ from $\mathfrak{C}$. A slightly more general setting has been taken into consideration by Marin \cite{M-10}, where $R=\mathbb{K}$, a field of characteristic zero such that each ordinary representation of $G$ is defined over $\mathbb{K}$ (for instance, any $\mathbb{K}$ containing the field of cyclotomic numbers) and $\sigma$ is the $\mathbb{K}$-linear extension of the map $g\mapsto\alpha(g)g^{-1}$, where $\alpha:G\rightarrow\mathbb{K}^\times$ is a multiplicative character of $G$. Here also $\mathbb{K} G^-_\sigma$ is shown to admit a similar kind of decomposition. Both the papers have used Frobenius-Schur Theory and concluded an analogue of Wedderburn decomposition for the corresponding Lie subalgebras. All these results motivated us to study these kinds of Lie algebras on a more general setting. We have considered the group algebra $FG$ of a finite group $G$ where $F$ is an algebraic extension of $\mathbb{Q}$ and its Lie subalgebra $FG^-_\sigma$ with respect to any arbitrary involution $\sigma$ of $FG$.

As a result of this study we have obtained the following results. All the notions of the statement will be recalled in section $2$.

\begin{thm}\label{main result 1}Let $FG$ be a group algebra of a finite group G over a field $F$ which is an algebraic extension of $\mathbb{Q}$ and $\sigma$ be an involution of $FG$. Then there exists an involution $\theta$ on $F$ and a right $F$-vector space $M$ which is also a left $FG$-module and a nonsingular hermitian or skew-hermitian form $h:M\times M\rightarrow F$ on $M$ (with respect to the involution $\theta$ on $F$) such that $\sigma$ is the adjoint involution of $h$ and
	\begin{equation}\label{main thm eqn}
	FG_\sigma^-=\left\{f\in FG\;|\;h(f(x),y)+h(x,f(y))=0\;\text{for all}\;x,y\in M\right\}.
	\end{equation}
\end{thm}
\vspace*{.5cm}
The $\theta$ and $M$ in the above statement have been constructed explicitly in this work (Section $3$). Our second main result is a decomposition theorem for the above Lie algebra of skew-symmetric elements.

\begin{thm}\label{main result 2} If $FG\simeq\bigoplus_{V\in \tilde{G}}End_{D}(V)$ as $F$-algebras, where $\tilde{G}$ denotes the set of all irreducible left $FG$-modules up to isomorphism with $D$ being division algebras over $F$ and the degree of the central simple algebra $End_D(V)$ over its center $Z$ is $n_V$, then $$FG_\sigma^-\simeq\bigoplus_{V\in\tilde{G}_o}\mathfrak{o}(V)\oplus\bigoplus_{V\in\tilde{G}_{sp}}\mathfrak{sp}(V)\oplus\bigoplus_{V\in\tilde{G}_{u}}{'}\mathfrak{gl}(V),$$
	as Lie algebras over $F$ where $\tilde{G}_o,\;\tilde{G}_{sp}$ and $\tilde{G}_u$ are subsets of $\tilde{G}$ such that
	\begin{eqnarray*}
		\tilde{G}_{o}&=&\{V\;|\;\sigma\left(End_D(V)\right)=End_D(V),\;dim_Z\left(End_D(V)\right)_\sigma^-=n_V(n_V-1)/2\},\\
		\tilde{G}_{sp}&=&\{V\;|\;\sigma\left(End_D(V)\right)=End_D(V),\;dim_Z\left(End_D(V)\right)_\sigma^-=n_V(n_V+1)/2,\\
		&\;& \qquad\qquad\qquad\qquad\qquad\qquad\qquad\qquad\qquad\qquad\qquad\qquad n_V\text{ is even}\},\\
		\tilde{G}_{u}&=&\{V\;|\;\sigma\left(End_D(V)\right)=End_D(V)\;or\;\sigma\left(End_D(V)\right)=End_{D^{op}}(V^{op}),\\
		&\;&\qquad\qquad\qquad\qquad\qquad\qquad\qquad\qquad\qquad dim_Z\left(End_D(V)\right)_\sigma^-=n^2_V\}.
	\end{eqnarray*}
	The prime signifies that there is just one summand $\mathfrak{gl}(V)$ for each pair $\{V,V^{op}\}$ from $\tilde{G}$.
\end{thm}

The manuscript has been divided into three sections. The second section recalls some basic definitions and results to be used in Section $3$. The third section is devoted to the proof of the main theorems, where we will first decompose $FG$ as a semisimple algebra with involution into a direct product of central simple algebras with involution. On each such central simple algebra, we will construct a nonsingular hermitian or skew hermitian form with the help of its involution. This nonsingular form will give rise to a Lie algebra on that component which we will show to be exactly equal to the image of the Lie algebra of skew-symmetric elements of $FG$ under the projection of $FG$ onto that component. Finally we will extend the results similarly to the whole of $FG$.


\section{Background}

Let the pair $(A,\sigma)$ denote an algebra $A$ with $1$ over a field $F$ with $\sigma$ being an involution on $A$ and is such that $F=\mathcal{Z}(A)\cap A_\sigma^+$, where $\mathcal{Z}(A)$ is the center of the algebra $A$. We will assume throughout this section that char $F\neq2$ and algebras are always finite dimensional over the corresponding fields. Note that the map $\sigma$ preserves the center $\mathcal{Z}(A)$. For a commutative algebra $K$ over $F$, we set $[K:F]=dim_FK$.  \\

As the involution $\sigma$ is an anti-automorphism with $\sigma^2=id$, the restriction of $\sigma$ to $\mathcal{Z}(A)$ is an automorphism which is either the identity or of order $2.$  Throughout this paragraph we will assume that $(A,\sigma)$ is an $F$-algebra with involution $\sigma$ such that $A$ has no non-trivial two sided ideal $I$ with $\sigma(I)=I$. With this assumption, either $\mathcal{Z}(A)=F$ or $\mathcal{Z}(A)$ is a quadratic \'{e}tale extension of $F$ (which means that $\mathcal{Z}(A)$ is either a field which is a separable quadratic extension of $F$ or $\mathcal{Z}(A)\simeq F\times F$). One can refer to the beginning of Section $3$, \cite{BBT-18} for this paragraph. One says that $(A,\sigma)$, or the involution $\sigma$, is of the \emph{first kind} or the \emph{second kind}, respectively, according to whether $[\mathcal{Z}(A):F]$ equals $1$ or $2$. As long as the center $\mathcal{Z}(A)$ of $A$ is a field, it follows that $A$ is central simple as a $\mathcal{Z}(A)$-algebra. This follows from the observation that if $I$ is a two sided ideal of $A$, then the two sided ideal $I+\sigma(I)$ is invariant under $\sigma$. Thus $I+\sigma(I)=0$ according to the assumption in the beginning of this paragraph. This implies that $\sigma(I)=-I=I$, which in turn implies that $I=0$. Thus $A$ has no non-trivial two sided ideal and hence is simple. However, if $(A,\sigma)$ is of the second kind, then as $[\mathcal{Z}(A):F]=2$, we may also have $\mathcal{Z}(A)=F\times F.$ \\

Involutions of the first kind can be divided into two categories. Let $(A,\sigma)$ be a central simple $F$-algebra of degree $n$ with involution of the first kind and let $L$ be a splitting field of $A$ $($that is, $L$ is a field containing $F$ such that $A_L=A\otimes_FL\simeq M_n(L)\;)$. Then $\sigma$ extends to an involution of the first kind $\sigma_L=\sigma\otimes Id_L$ on $A_L$. Let $V$ be an $L$-vector space of dimension $n$. There is a nonsingular symmetric or skew-symmetric bilinear form $b$ on $V$ and $(A_L,\sigma_L)\simeq\left(End_L(V),\sigma_b\right)$, where $\sigma_b$ is the adjoint involution (for the definition of adjoint involution see Lemma \ref{Adjoint Involution} below) of the form $b$. (Proposition $2.1$, \cite{KMRT}). If $b$ is symmetric, we say the involution $\sigma$ is of \emph{orthogonal type}. If $b$ is skew-symmetric, then $\sigma$ is said to be of \emph{symplectic type}. As a parallel terminology involutions of the second kind are also sometimes said to be of \emph{unitary type}.\\

The following definition is $1.2.7$ of \cite{K}.

\begin{defn}\label{hyperbolic ring}For any ring $R$, let $R^{op}$ be the opposite ring, that is, the same additive group with reversed multiplication $$a^{op}.b^{op}=(ba)^{op},$$ where $a^{op}$ stands for $a$ as an element of $R^{op}$. Involutions on $R$ are isomorphisms $\phi: R\rightarrow R^{op}$ of rings such that $\phi^2=1_R.$ For any ring $R$, the product $R\times R^{op}$ has an involution $$(a,b^{op})\mapsto(b,a^{op}).$$ This product is called the \emph{hyperbolic ring} of $R$ and is denoted by $H(R).$ This involution is called the \emph{exchange involution}.
\end{defn}

The result below follows from Proposition $2.14$ and its proof in \cite{KMRT}. Some of the notations used are from definition \ref{hyperbolic ring} above.

\begin{lem}\label{exachange involution} Let $A$ be a semisimple $F$-algebra of the form $A_1\times A_2$, where $A_1$, $A_2$ are central simple $F$-algebras and $A$ is endowed with an involution $\sigma$ of the second kind. Then $(A,\sigma)=(A_1\times A_2,\;\sigma)\simeq (A_1\times {A_1}^{op},\;\varepsilon)$ where $\varepsilon$ is the exchange involution. That is, $(A,\sigma)\simeq \big(H(A_1),\varepsilon\big).$
\end{lem}

The following is example $1.2.8$ of Chapter $I$ of \cite{K}.

\begin{prop}\label{semisimple} Let $A$ be a semisimple $F$-algebra with involution $\sigma$. We can decompose $A$ as a product $A_1\times\cdots\times A_n\times B_1\times\cdots\times B_m,$ where the $A_i$ are simple and $\sigma(A_i)=A_i$ and the $B_j$ are products of two simple algebras $A'_j\times A''_j$ with $\sigma(A'_j)=A''_j.$ The rings $B_j$ are isomorphic to hyperbolic rings, $B_j\simeq H(A'_j).$ 
\end{prop}
\begin{proof}The proof follows from the fact that $\sigma$ being an involution, when restricted to any one of the simple components in the Wedderburn decomposition of the semisimple algebra $A$, will either take that component to itself or to another simple component due to simplicity of the domain. Thus by rearranging the simple components in the Wedderburn decomposition of $A$ according to whether $\sigma$ sends the simple component to itself or to another simple component, we get the decomposition as mentioned in the statement. Clearly, $\sigma$ restricted to $B_j$'s are involutions of the second kind where the center $\mathcal{Z}(B_j)=\mathcal{Z}(A'_j)\times\mathcal{Z}(A''_j)$ is a product of two factors $K\times K$ where $K$ is a finite field extension of $F$. The fact that $B_j\simeq H(A'_j)$ follows from Lemma \ref{exachange involution}.
\end{proof}

Let us recall the definition of a hermitian form with respect to an involution on a finite dimensional central simple algebra. (See $4.\text{A}$, \cite{KMRT}).

\begin{defn}[\textbf{Hermitian Forms}]\label{Hermitian Definition}Let $E$ be a central simple algebra over a field $F$ with char $F\neq2$ and let $M$ be a finitely generated right $E$-module. Suppose that $\theta:E\rightarrow E$ is an involution (of any kind) on $E$. A \emph{hermitian form} on $M$ (with respect to the involution $\theta$ on $E$) is a bi-additive map $h:M\times M\rightarrow E$ such that:
	\begin{enumerate}
		\item $h(x\alpha,y\beta)=\theta(\alpha)h(x,y)\beta$ for all $x,y\in M$ and $\alpha,\beta\in E,$
		\item $h(y,x)=\theta\left(h(x,y)\right)$ for all $x,y\in M$.
	\end{enumerate}
	If $(2)$ is replaced by $h(y,x)=-\theta\left(h(x,y)\right)$ for all $x,y\in M$, the form $h$ is called \emph{skew-hermitian}. If $E=F$ and $\theta=Id$, hermitian (resp. skew-hermitian) forms are the symmetric (resp. skew-symmetric) bilinear forms on the finite dimensional vector space $M$ over $F$. \\
	The hermitian or skew-hermitian form $h$ on the right $E$-module $M$ is called \emph{nonsingular} if the only element $x\in M$ such that $h(x,y)=0$ for all $y\in M$ is $x=0$.
\end{defn}

The following is Proposition $4.1$ in \cite{KMRT}. The notations used are from the above definition \ref{Hermitian Definition}.

\begin{lem}[\textbf{Adjoint Involution}]\label{Adjoint Involution}For every nonsingular hermitian or skew-hermitian form $h$ on $M$, there exists a unique involution $\sigma_h$ on $End_E(M)$ such that $\sigma_h(\alpha)=\theta(\alpha)$ for all $\alpha\in F$ and $$h\left(x,f(y)\right)=h\left(\sigma_h(f)(x),y\right)\;\text{for}\;x,y\in M,\;f\in End_E(M).$$ The involution $\sigma_h$ is called the adjoint involution with respect to $h.$
\end{lem}

Let $Z/F$ be a finite extension of fields, $E$ a central simple algebra over $Z$ and $T$ a central simple $F$-algebra contained in $E$. Suppose that $\theta$ is an involution (of any kind) on $E$ which preserves $T$.

\begin{defn}[\textbf{Involution trace}]\label{involution trace}  An $F$-linear map $s:E\rightarrow T$ is called an \emph{involution trace} if it satisfies the following conditions:
	\begin{enumerate}
		\item $s(t_1xt_2)=t_1s(x)t_2$ for all $x\in E$ and $t_1,t_2\in T;$
		\item $s(\theta(x))=\theta(s(x))$ for all $x\in E;$
		\item if $x\in E$ is such that $s(\theta(x)y)=0$ for all $y\in E$, then $x=0.$
	\end{enumerate}
\end{defn}

\begin{ex}[\textbf{Reduced Trace}]\label{reduced trace} 
	In the above definition \ref{involution trace}, if $T=F=Z$, the reduced trace $Trd_{E/F}:E\rightarrow F$ is an involution trace. The first two conditions follow from $2.2$ and $2.16 $ in \cite{KMRT} and the third condition follows from the fact that the bilinear reduced trace form $T_{E/F}:E\times E\rightarrow F$ defined by: $T_{E/F}(x,y)=Trd_{E/F}(xy)$ for $x,y\in E$ is nonsingular (Theorem 9.9, \cite{R}).
\end{ex}

\begin{ex}[\textbf{Linear Forms}]\label{linear forms}In definition \ref{involution trace}, if $E=Z$ and $T=F$, every nonzero linear map $l:Z\rightarrow F$ which commutes with $\theta$ is an involution trace. 
	Nonzero linear forms $Z\rightarrow F$ always exist if the extension $Z/F$ is separable. For example, $Trd_{Z/F}$ will be an involution trace (that is, nonzero linear form) for the trivial involution if and only if $Z/F$ is a separable field extension $($\cite{K}, Chapter $I$, $7.3.2)$.
\end{ex}


\section{Proof of the main theorems}

We begin with the Wedderburn Artin decomposition of $FG$ with involution.

\begin{prop}Let $G$ be a finite group and $\sigma$ be an involution of $FG$, where $F$ is an algebraic extension of $\mathbb{Q}$. Then as a semisimple algebra with involution, $FG$ decomposes into:
	\begin{eqnarray}\label{FG decomposition}
	\nonumber(FG,\sigma)&\simeq&(M_{a_1}(D_1),\sigma_{a_1})\times\cdots\cdots\cdots\times(M_{a_r}(D_r),\sigma_{a_r})\times\\
	\nonumber&\;&\left(M_{b_1}(D'_1)\times M_{b_1}(D'_1)^{op},\sigma_{b_1}\right)\times\cdots\cdots\cdots\cdots\cdots\\
	&\;&\qquad\qquad\qquad\cdots\times\left(M_{b_t}(D'_t)\times M_{b_t}(D'_t)^{op},\sigma_{b_t}\right),
	\end{eqnarray}
	where $\sigma_{a_i}$ are involutions of $M_{a_i}(D_i)$ and $\sigma_{b_j}$ are involutions of $M_{b_j}(D'_j)\times M_{b_j}(D'_j)^{op}$ such that $\sigma_{b_j}\big(M_{b_j}(D'_j)\times\{0\}\big)= \{0\}\times M_{b_j}(D'_j)^{op},$ for $1\leq i\leq r$ and $1\leq j\leq t.$ 	
\end{prop}

\begin{proof}As $FG$ is semisimple, by Wedderburn Structure Theorem, $FG$ is isomorphic as $F$-algebra to a product of simple algebras in the following way:
	$$FG\simeq M_{n_1}(D_1)\times M_{n_2}(D_2)\times\cdots\times M_{n_k}(D_k),$$ where $D_i$ are finite dimensional division algebras over $F$ such that the center $\mathcal{Z}\left(M_{n_i}(D_i)\right)=\mathcal{Z}(D_i)=\mathcal{Z}_i$ is a finite field extension over $F.$ Thus $M_{n_i}(D_i)$ are central simple algebras over $\mathcal{Z}_i.$  Now, by Proposition \ref{semisimple},  $(FG,\sigma)$ decomposes as $F$-algebra into
	\begin{eqnarray}\label{KG decomposition}
	\nonumber FG&\simeq&M_{a_1}(D_1)\times\cdots\cdots\times M_{a_r}(D_r)\times M_{b_1}(D'_1)\times M_{b_1}(D''_1)\\
	&\;&\times\cdots\cdots\times M_{b_t}(D'_t)\times M_{b_t}(D''_t),
	\end{eqnarray}
	where $\sigma\left(M_{a_i}(D_i)\right)=M_{a_i}(D_i)$ and $\sigma\left(M_{b_j}(D'_j)\right)=M_{b_j}(D''_j)$, for $1\leq i\leq r,\;1\leq j\leq t.$ 
	Again, by Proposition \ref{semisimple}, 
	$$(M_{b_j}(D'_j)\times M_{b_j}(D''_j),\sigma)\simeq\big(M_{b_j}(D'_j)\times M_{b_j}(D'_j)^{op},\sigma_{b_j}\big),$$ where $\sigma_{b_j}$ is the switch involution. Taking $\sigma$ restricted to $M_{a_i}(D_i)$ as $\sigma_{a_i}$ for $1\leq i\leq r$, the result follows. 
\end{proof}

All the constructions and computations henceforth will be written down for the component  $M_{n_i}(D_i)$ only in the Wedderburn decomposition of $FG$ as they are similar for the direct product $M_{b_j}(D'_j)\times M_{b_j}(D'_j)^{op}$.\\

Let us fix $i$ and denote $n_i$, $D_i$ and $\sigma_{n_i}$ as $n$, $D$ and $\sigma$ respectively. We concentrate on $(M_n(D),\sigma)$. Now, $M_n(D)$ is isomorphic as $F$-algebra to $End_D(D^n)$, where $D^{n}$ can be viewed as a right vector space over $D.$  As $M_n(D)$ is Brauer equivalent to $D$, by a well known theorem by A.A. Albert (\cite{KMRT}, Theorem $3.1$), there exists an involution, say $\theta$ on $D$ which is of the same kind as $\sigma$. Thus, according to Theorem $4.2$, \cite{KMRT}, there exists a nonsingular hermitian or skew-hermitian form $h:D^n\times D^n\rightarrow D$ with respect to $\theta$ such that $\sigma$ is the adjoint involution of $h$ (recall definition \ref{Adjoint Involution}), that is, $\sigma=\sigma_h$ and it satisfies $h(x,f(y))=h\left(\sigma_h(f)(x),y\right)$ for all $x,y\in D^n$ and $f\in End_D(D^n).$\\

Let $Z$ denote the center of $D$. Then $D$ is a finite dimensional central simple algebra over $Z$ which is a finite field extension over $F$ and $F$ again is an algebraic extension of $\mathbb{Q}$. Hence $Z/F$ is separable. By definition, $\theta$ fixes $F$ pointwise. Thus we can choose a nonzero linear form $l:Z\rightarrow  F$ (by Example \ref{linear forms}) that commutes with $\theta$ restricted to $Z$. Let $Trd_{D/Z}:D\rightarrow Z$ be the reduced trace of the central simple algebra $D$ over $Z$ which is an involution trace by Example \ref{reduced trace}. Now take the composition of these two maps to define the map $s: D\rightarrow F$ where $s=l\circ Trd_{D/Z}$.\\

\begin{lem}  With the notations as above, $s$ is an involution trace. \end{lem}

\begin{proof}We check the conditions one by one.
	\begin{enumerate}
		\item $s(\alpha_1x\alpha_2)=(l\circ Trd_{D/Z})(\alpha_1x\alpha_2)=\alpha_1s(x)\alpha_2$ for all $\alpha_1,\alpha_2\in F\subseteq Z$ and $x\in D$ follows directly from the definition of $l$ and $Trd_D$.
		\item $s(\theta(x))=(l\circ Trd_{D/Z})(\theta(x))=l(\theta(Trd_{D/Z}(x)))=\theta\left(l(Trd_{D/Z}(x))\right)=\theta(s(x))$ for all $x\in D$ as $Trd_{D/Z}$ is an involution trace and $l$ is a nonzero linear form that commutes with $\theta$ restricted to $Z$.
		\item Let $x\in D$ be such that $s(\theta(x)y)=0$ for all $y\in D$. This means $l\left(Trd_{D/Z}(\theta(x)y)\right)=0$ for all $y\in D$. Since $l$ is a nonzero linear form on a finite dimensional vector space $Z$ over $F$, we must have $Trd_{D/Z}(\theta(x)y)=0$ for all $y\in D$, which in turn implies that $x=0$ as $Trd_{D/Z}$ is an involution trace. Thus $s$ satisfies the third condition as well for being an involution trace.
	\end{enumerate}
\end{proof}
Now, applying Proposition $4.7$ in \cite{KMRT} to $E=D,\;T=F$ and $M=D^n$, we get a nonsingular hermitian or skew hermitian form (according as $h$) with respect to $\theta$ on $F$. The nonsingular form can be written as the following: 
$$s_*(h):D^n\times D^n\rightarrow F$$
such that $$s_*(h)(x,y)=s(h(x,y))\;\text{for}\;x,y\in D^n.$$ Also, $D^n$ is a right vector space over $F$ and $End_D(D^n)\subset End_F(D^n).$ The adjoint involution $\sigma_{s_*(h)}$ on $End_F(D^n)$ extends the adjoint involution $\sigma=\sigma_h$ on $End_D(D^n)$, that is, $$\left(End_D(D^n),\sigma_h\right)\subset\left(End_F(D^n),\sigma_{s_*(h)}\right).$$

In view of the above discussion Eq. \ref{FG decomposition} can be written as:
\begin{eqnarray*}
	(FG,\sigma)&\simeq&\Pi_{i=1}^r\left(End_{D_i}(D_i^{a_i}),\sigma_{a_i}\right)\times\\
	&\;&\qquad\qquad\Pi_{j=1}^t\left(End_{D'_j\times {D'_j}^{op}}\left({D'_j}^{b_j}\times\left({D'_j}^{op}\right)^{b_j}\right),\sigma_{b_j}\right)\\
	&\subset&\Pi_{i=1}^r\left(End_F(D_i^{a_i}),\sigma_{s_*(h_{a_i})}\right)\times\\
	&\;&\qquad\qquad\Pi_{j=1}^t\left(End_{F\times F}({D'_j}^{b_j}\times\left({D'_j}^{op}\right)^{b_j},\sigma_{s_*(h_{b_j})}\right),
\end{eqnarray*}
where $\sigma_{a_i}=\sigma_{h_{a_i}}$ are the adjoint involutions of the nonsingular hermitian or skew hermitian forms $h_{a_i}:D_i^{a_i}\times D_i^{a_i}\rightarrow D_i$ with respect to the involution $\theta_{a_i}$ $($which is of the same kind as $\sigma_{h_{a_i}})$ on $D_i$ and $\sigma_{b_j}=\sigma_{h_{b_j}}$ are the adjoint involutions of the nonsingular hermitian forms $h_{b_j}: \left({D'_j}^{b_j}\times\left({D'_j}^{op}\right)^{b_j}\right)\times\left({D'_j}^{b_j}\times\left({D'_j}^{op}\right)^{b_j}\right)\rightarrow D'_j\times {D'_j}^{op}$ with respect to the involution (of second kind) $\theta_{b_j}$ on $D'_j\times {D'_j}^{op}$.\\


The set $\{x\in End_D(D^n)\;|\;h(xu,v)+h(u,xv)=0\;\text{for all}\;u,v\in D^n\}$, where $h:D^n\times D^n\rightarrow D$ is a nonsingular hermitian or skew-hermitian form on $D^n$ with respect to some involution $\theta$ (of any kind) on $D$, is a Lie algebra over $F$.\\
Now, $D^n$ is a simple $FG$-module. Thus each $f\in FG$ acts on $D^n$. Thus the projection, say $f_i$ of $f$ on any of the factors $M_n(D)\simeq End_D(D^n)$ of $FG$ can be viewed as an element in $End_D(D^n)$. Now we prove the following characterization of the projection of the skew-symmetric elements on one of the summands of $FG$.

\begin{lem}\label{image} The image of $FG_\sigma^-$ under the projection of $FG$ onto $M_n(D)$ consists of all linear transformations $f_i$ in $End_D(D^n)$ such that $s_*(h)(f_i(x),y)+s_*(h)(x,f_i(y))=0$ for all $x,y\in D^n$.
\end{lem}

\begin{proof}Recall that $\sigma=\sigma_h$ is the adjoint involution of the nonsingular hermitian or skew-hermitian form $h$ on $D^n$ with respect to $\theta$, that is, $h(x,f(y))=h\left(\sigma_h(f)(x),y\right)$ for all $x,y\in D^n$ and $f\in End_D(D^n).$ If $f_i$ belongs to the image of $FG_\sigma^-$ under the projection of $FG$ onto $M_n(D)$, then $\sigma(f_i)=-f_i$, by definition of $FG_\sigma^-$. The involution $\sigma_{s_*(h)}$ of $End_F(D^n)$ extends the involution $\sigma_h$ on $End_D(D^n)$. That is, $\sigma_{s_*(h)}$ restricted to $End_D(D^n)$ is $\sigma_h$. Thus, as $f_i\in End_D(D^n)$, we have $\sigma_{s_*(h)}(f_i)=\sigma_h(f_i)=-f_i$. This implies $h(x,f_i(y))=h\left(\sigma_h(f_i)(x),y\right)=-h(f_i(x),y)$. So we get $$\sigma(f_i)=-f_i\;\Leftrightarrow\;s_*(h)(f_i(x),y)+s_*(h)(x,f_i(y))=0\;\text{for}\;x,y\in D^n.$$\end{proof}

We are now in a position to prove our main theorems. Recall definitions \ref{Hermitian Definition} and \ref{Adjoint Involution} for the proof of Theorem \ref{main result 1}.

\begin{proof}[Proof of Theorem \ref{main result 1}]
	
	Let $M=\Pi_{i=1}^{r}D_i^{a_i}\times\Pi_{j=1}^{t}\left({D'_j}^{b_j}\times\left({D'_j}^{op}\right)^{b_j}\right)$. Then $M$ is a right vector space over $F$ and also a left $FG$-module. Define $h:M\times M\rightarrow F$ by
	$$h\left(\sum_{i=1}^rx_i+\sum_{j=1}^tx'_j,\sum_{i=1}^ry_i+\sum_{j=1}^ty'_j\right)=\bigoplus_{i=1}^r s_*(h_{a_i})(x_i,y_i)\oplus \bigoplus_{j=1}^t s_*(h_{b_j})(x'_j,y'_j),$$ where $x_i,y_i\in D_i^{a_i}$ and $x'_j,y'_j\in {D'_j}^{b_j}\times\left({D'_j}^{op}\right)^{b_j}.$
	Then $h$ is a nonsingular form (hermitian or skew-hermitian) on $M$ with respect to the involution $\theta$ on $F$, where $\theta=\bigoplus_{i=1}^r\theta_{a_i}\oplus \bigoplus_{j=1}^t\theta_{b_j}$ with $\theta_{a_i}$ being $F$-restrictions of the corresponding involutions on $D_i$ for $1\leq i\leq r$ and $\theta_{b_j}$ being $F\times F$-restrictions of the corresponding involutions on $D'_j\times {D'_j}^{op}$ for $1\leq j\leq t.$ Also we have the given involution $\sigma$ on $FG$ is the adjoint involution of $h$ as by construction $\sigma$ is expressed as the sum of adjoint involutions on each simple component of $FG$ with respect to the corresponding nonsingular hermitian or skew-hermitian form on that component. With the above notations, we will show that Equation \ref{main thm eqn} of Theorem \ref{main result 1} holds.\\
	
	Let $f\in FG_\sigma^-\subset(FG,\sigma)$. Thus $f=\sum_{i=1}^rf_i+\sum_{j=1}^tf'_j$, where each $f_i$ $($resp. $f'_j)$ belongs to the image of $FG_\sigma^-$ under the projection of $FG$ onto $End_{D_i}(D_i^{a_i})$ \bigg(resp. $End_{D'_j\times {D'_j}^{op}}\left({D'_j}^{b_j}\times\left({D'_j}^{op}\right)^{b_j}\right)$\bigg). Writing $x=\sum_{i=1}^rx_i+\sum_{j=1}^tx'_j$ and $y=\sum_{i=1}^ry_i+\sum_{j=1}^ty'_j$, and using Lemma \ref{image}, we have:
	\begin{flalign*}
	&\qquad\qquad\;\sigma(f)=-f\\
	&\Longleftrightarrow \qquad \bigoplus_{i=1}^r\sigma_{h_{a_i}}(f_i)\oplus\bigoplus_{j=1}^t\sigma_{h_{b_j}}(f'_j)=-\sum_{i=1}^rf_i-\sum_{j=1}^tf'_j\\
	&\Longleftrightarrow \qquad\bigoplus_{i=1}^r\left(s_*(h_{a_i})\left(f_i(x_i),y_i\right)+s_*(h_{a_i})\left(x_i,f_i(y_i)\right)\right)\\
	&\qquad\qquad\oplus\bigoplus_{j=1}^t\left(s_*(h_{b_j})\left(f'_j(x'_j),y'_j\right)+s_*(h_{b_j})\left(x'_j,f'_j(y'_j)\right)\right)=0\\
	&\Longleftrightarrow \qquad\; h(f(x),y)+h(x,f(y))=0.
	\end{flalign*}
\end{proof}

\begin{proof}[Proof of Theorem \ref{main result 2}]
	
	It follows from Lemma \ref{image} that the image of $FG_\sigma^-$ under the projection of $FG$ onto $M_n(D)$ is a Lie algebra. Let $V=D^n$. Then V is an irreducible left $FG$-module. From Theorem $4.2$, \cite{KMRT} it follows that if $\sigma$ is of the first kind on $End_D(V)$, then $\sigma$ restricted to $End_D(V)$ is orthogonal (resp. symplectic) if the involution $\theta$ on $D$ is orthogonal (resp. symplectic) provided the corresponding nonsingular form $h:V\times V\rightarrow D$ with respect to $\theta$ on $D$  is hermitian. If the nonsingular form $h$ is skew-hermitian, then $\sigma$ restricted to $End_D(V)$ is orthogonal (resp. symplectic) if $\theta$ is symplectic (resp. orthogonal). Let the degree of the central simple algebra $End_D(V)$ over its center $Z$ be $s$. By Proposition $2.6$ and Proposition $2.17$ of \cite{KMRT} the dimension of the Lie algebra $End_D(V)_\sigma^-$ is $s(s-1)/2$ if $\sigma$ restricted to it is orthogonal and the dimension is $s(s+1)/2$ if $\sigma$ restricted to it is symplectic. Whenever the involution $\sigma$ is symplectic, $s$ is even. If $\sigma$ is of the second kind (that is of unitary type) on $End_D(V)$ or $End_{D\times D^{op}}(V\times V^{op})$, then the dimension is $s^2$. Denoting the Lie algebras $End_D(V)_\sigma^-$ in the orthogonal, symplectic and unitary cases as $\mathfrak{o}(V)$, $\mathfrak{sp}(V)$ and $\mathfrak{gl}(V)$ respectively, we get the decomposition structure of $FG_\sigma^-$ as proposed.
\end{proof}

\bibliographystyle{alpha}
\bibliography{OneBibToRuleThemAll}


\end{document}